\documentclass[12pt]{amsproc}

\usepackage{setspace}
\usepackage{latexsym}
\usepackage{amsmath, amssymb, amsfonts, amsthm}
\usepackage{float}
\usepackage{color}
\usepackage{mathrsfs}
\usepackage{hyperref} 
\hypersetup{citecolor=red, linkcolor=blue, colorlinks=true}

\usepackage{bm}
\usepackage[margin=1in]{geometry}
\usepackage{url}

\newtheorem{thm}{Theorem}[section]
\newtheorem{lem}[thm]{Lemma}
\newtheorem{prop}[thm]{Proposition}

\newtheorem{func}[thm]{Function}
\newtheorem{defi}[thm]{Definition}

\newtheorem{rem}[thm]{Remark}
\newtheorem{calc}[thm]{Calculation}

\renewcommand\le{\leqslant}

\newcommand\PSL{{\rm PSL}}
\newcommand\PGL{{\rm PGL}}
\newcommand\GAP{\textbf{{\rm GAP}}}
\newcommand\C{{\mathbb{C}}}
\newcommand\N{{\mathbb{N}}}

\newcommand\GF{{\rm GF}}

\newcommand\PG{{\rm PG}}
\newcommand{\calA}{\mathcal{A}}

\newcommand{\calP}{\mathcal{P}}

\newcommand\Q{{\mathbb{Q}}}
\newcommand{\calQ}{\mathcal{Q}}
\newcommand{\Z}{\mathbb{Z}}
\newcommand{\Ker}{{\rm Ker}}
\newcommand{\bbK}{\mathbb{K}}

\title{Transitive PSL(2,11)-invariant $k$-arcs in PG(4,q)}

\author{Torger Olson}

\author{Eric Swartz}
\address{Department of Mathematics, College of William \& Mary, P.O. Box 8795, Williamsburg, VA 23187-8795, USA}
\email{tjolson@email.wm.edu, easwartz@wm.edu}
\thanks{This research was partially supported by the NSF EXTREEMS-QED grant DMS-1331921.}

\begin{document}

\begin{abstract}
A \textit{k}-arc in the projective space ${\rm PG}(n,q)$ is a set of $k$ projective points such that no subcollection of $n+1$ points is contained in a hyperplane. In this paper, we construct new $60$-arcs and $110$-arcs in ${\rm PG}(4,q)$ that do not arise from rational or elliptic curves. We introduce computational methods that, when given a set $\mathcal{P}$ of projective points in the projective space of dimension $n$ over an algebraic number field $\mathcal{Q}(\xi)$, determines a complete list of primes $p$ for which the reduction modulo $p$ of $\mathcal{P}$ to the projective space ${\rm PG}(n,p^h)$ may fail to be a $k$-arc. Using these methods, we prove that there are infinitely many primes $p$ such that ${\rm PG}(4,p)$ contains a ${\rm PSL}(2,11)$-invariant $110$-arc, where ${\rm PSL}(2,11)$ is given in one of its natural irreducible representations as a subgroup of ${\rm PGL}(5,p)$. Similarly, we show that there exist ${\rm PSL}(2,11)$-invariant $110$-arcs in ${\rm PG}(4,p^2)$ and ${\rm PSL}(2,11)$-invariant $60$-arcs in ${\rm PG}(4,p)$ for infinitely many primes $p$. 
\end{abstract}

\maketitle

\section{Introduction}

Let $q$ be a prime power.  A \textit{k-arc} in the projective space $\PG(n,q)$ is a set of $k$ projective points such that no subcollection of $n+1$ points is contained in a hyperplane, and such a $k$-arc is said to be \textit{complete} if it is not contained in a $(k+1)$-arc.

A motivation for studying $k$-arcs in projective spaces comes from Coding Theory and, in particular, the study of \textit{maximal distance separable codes}, commonly referred to as \textit{M.D.S. codes}, which are codes with the greatest error correcting capability.  As it turns out, M.D.S. codes and $k$-arcs are equivalent objects; see \cite{thas}.  For more on the connection between arcs and secret sharing schemes, see, for instance, \cite{JMO}, \cite{codes}, and \cite{SJM}.  For these reasons, in recent years there has been great interest in constructing new arcs in projective spaces; see \cite{BGP}, \cite{GKMP}, \cite{IK}, \cite{KP}, \cite{P}, \cite{PS}, \cite{S}.  Moreover, the automorphism group of a code can be used to decrease the computational complexity of encoding and decoding \cite{TvG}, so finding new examples of codes with large automorphism groups is useful in practice.

In this paper, we study $k$-arcs in $4$-dimensional projective space, about which relatively little seems to be known, especially when the codes do not arise from elliptic or rational curves.  In particular, we study $k$-arcs that are invariant under the action of $\PSL(2,11)$, which has a $5$-dimensional irreducible representation over $\GF(q)$ whenever $q^5 \equiv 1 \pmod {11}$; see Section \ref{sect:5dimrep} for details.  Our main results are the following, which show the existence of $\PSL(2,11)$-transitive $60$-arcs and $110$-arcs in infinitely many $4$-dimensional projective spaces and detail precisely when they occur.

\begin{thm}
\label{thm:110prime}
 There are infinitely many primes $p$ such that there exists a $\PSL(2,11)$-transitive $110$-arc in $\PG(4,p)$, and a list of all primes $p$ such that $p^5 \equiv 1 \pmod {33}$ and $\PG(4,p)$ may not contain a $\PSL(2,11)$-transitive $110$-arc is known.  Explicitly, there are exactly $2728$ primes in this list, and $\PG(4,p)$ contains a $\PSL(2,11)$-transitive $k$-arc whenever $p^5 \equiv 1 \pmod {33}$ and $p > 5373427$.  In particular, the primes $p$ such that $3 < p < 65000$ and $\PG(4,p)$ contains a $\PSL(2,11)$-transitive $110$-arc are $26029,$
 $26437,$
 $27127,$
 $27481,$
 $28081,$
 $28759,$
 $29401,$
 $30259,$
 $31069,$
 $32257,$
 $32803,$
 $33247,$
 $33301,$
 $34159,$
 $34543,$
 $34747,$
 $35797,$
 $35869,$
 $36061,$
 $36217,$
 $37339,$
 $37579,$
 $38239,$
 $38281,$
 $38317,$
 $38371,$
 $38449,$
 $39301,$
 $39439,$
 $40093,$
 $40099,$
 $40357,$
 $40423,$
 $40771,$
 $40903,$
 $41023,$
 $41143,$
 $41221,$
 $41299,$
 $41737,$
 $41809,$
 $41911,$
 $41959,$
 $42013,$
 $42397,$
 $42409,$
 $42463,$
 $42751,$
 $42901,$
 $43399,$
 $43759,$
 $43789,$
 $44203,$
 $44383,$
 $44797,$
 $44851,$
 $44983,$
 $45433,$
 $45631,$
 $45841,$
 $46171,$
 $46399,$
 $46663,$
 $46861,$
 $47221,$
 $47287,$
 $47389,$
 $47491,$
 $48409,$
 $48541,$
 $49339,$
 $49393,$
 $49417,$
 $49603,$
 $49921,$
 $49999,$
 $50671,$
 $50821,$
 $50989,$
 $51439,$
 $51613,$
 $51859,$
 $52177,$
 $52453,$
 $52501,$
 $52561,$
 $52567,$
 $52783,$
 $52837,$
 $52957,$
 $53047,$
 $53089,$
 $53101,$
 $53161,$
 $53551,$
 $53593,$
 $53617,$
 $53857,$
 $53881,$
 $53887,$
 $54037,$
 $54151,$
 $54367,$
 $54499,$
 $54547,$
 $54679,$
 $54829,$
 $54949,$
 $54979,$
 $55201,$
 $55213,$
 $55411,$
 $55639,$
 $55807,$
 $55837,$
 $55903,$
 $55933,$
 $56101,$
 $56167,$
 $56269,$
 $56467,$
 $56629,$
 $56809,$
 $56827,$
 $57601,$
 $57667,$
 $57847,$
 $58129,$
 $58237,$
 $58243,$
 $58369,$
 $58567,$
 $58573,$
 $58789,$
 $58921,$
 $58963,$
 $59119,$
 $59167,$
 $59233,$
 $59557,$
 $59779,$
 $59863,$
 $59887,$
 $59929,$
 $60223,$
 $60589,$
 $60679,$
 $60901,$
 $61099,$
 $61141,$
 $61381,$
 $61417,$
 $61561,$
 $61603,$
 $61681,$
 $61933,$
 $62011,$
 $62131,$
 $62533,$
 $62617,$
 $62683,$
 $62701,$
 $62731,$
 $63127,$
 $63277,$
 $63331,$
 $63361,$
 $63391,$
 $63409,$
 $63463,$
 $63559,$
 $63853,$
 $64381,$
 $64513,$
 $64579,$
 $64663,$
 and $64783.$  
\end{thm}

The complete list of all $2767$ primes $p^5 \equiv 1 \pmod {33}$ such that $p \le 5373427$ for which there may not be a $\PSL(2,11)$-transitive $110$-arc in $\PG(4,p)$ can be found online at\\ \url{http://www.math.wm.edu/~eswartz/badprimes110}.

\begin{thm}
\label{thm:110prime2}
 There are infinitely many primes $p$ such that there exists a $\PSL(2,11)$-transitive $110$-arc in $\PG(4,p^2)$ (and not in $\PG(4,p)$), and a list of all primes $p$ such that $p^5 \equiv 23 \pmod {33}$ and $\PG(4,p^2)$ does not contain a $\PSL(2,11)$-transitive $110$-arc is known.  Explicitly, $\PG(4,p^2)$ (and not $\PG(4,p)$) contains a $\PSL(2,11)$-transitive $110$-arc if and only if $p = 311,$
 $317,$
 $389,$
 $401,$
 $419,$
 $443,$
 $449,$
 $467,$
 $509,$
 $521,$
 $587,$
 $599,$
 $617,$
 $641,$
 $653,$
 $719,$
 $773,$
 $839,$
 $881,$
 $911,$
 $947,$
 $977,$
 $983,$
 $1013,$
 $1049,$
 $1061,$
 $1103,$
 $1109,$
 $1181,$
 $1193,$
 $1259,$
 $1277,$
 $1301,$
 $1307,$
 $1367,$
 $1373,$
 $1409,$
 $1433,$
 $1499,$
 $1511,$
 $1523,$
 $1571,$
 $1607,$
 $1637,$
 $1697,$
 $1709,$
 $1721,$
 $1787,$
 $1901,$
 $1907,$
 $1973,$
 $2003,$
 $2027,$
 $2039,$
 $2069,$
 $2099,$
 $2237,$
 $2267,$
 $2297,$
 $2333,$
 $2357,$
 $2381,$
 $2399,$
 $2423,$
 $2447,$
 $2531,$
 $2579,$
 $2621,$
 $2633,$
 $2663,$
 $2687,$
 $2693,$
 $2699,$
 $2711,$
 $2729,$
 $2753,$
 $2777,$
 $2819,$
 $2843,$
 $2861,$
 $2897,$
 $2909,$
 $2927,$
 $2957,$
 $2963,$
 $3023,$
 $3041,$
 $3083,$
 $3089,$
 $3191,$
 $3221,$
 $3257,$
 $3323,$
 $3347,$
 $3359,$
 $3371,$
 $3389,$
 $3413,$
 $3491,$
 $3557,$
 $3617,$
 $3623,$
 $3677,$
 $3701,$
 $3719,$
 $3767,$
 $3821,$
 $3833,$
 $3851,$
 $3881,$
 or $p^5 \equiv {23} \pmod {33}$ and $p > 3917$. 
\end{thm}

\begin{thm}
\label{thm:60prime}
 There are infinitely many primes $p$ such that there exists a $\PSL(2,11)$-transitive $60$-arc in $\PG(4,p)$.  Explicitly, there there is a $\PSL(2,11)$-transitive $60$-arc in $\PG(4,p)$ if and only if $p = 1277$, $1783,$ $2069,$ $2333,$ $2399,$ $2861,$ $2971,$ $3169,$ $3499,$ $3631,$ $4027,$ $4159,$ $4357,$ $4423,$ $4621$, $4643,$ $4951,$ $4973,$ $5039,$ $5171,$ $5237,$ $5281$, $5303,$ $5347$, $5413$, $5479$, $5501$, $5743$, $5897$, $6007$, $6029$, $6073$, $6271,$ $6337$, $6359$, $6469$, $6689$, $6733$, $6997$, $7019$, $7129$, $7151$, $7283$, $7349$, $7393$, $7459$, $7481$, $7547$, $7723$, $7789$, $8009$, $8053$, $8273$, $8317$, $8537$, $8581$, $8647$, $8669$, $8713$, $8779$, $8867$, $8933$, $8999$, $9043$, $9109$, $9241$, $9439$, $9461$, $9769$, $9791$, $9857$, $9901$, $9923$, $9967$, $10099$, $10253$, $10429$, $10627$, $10781$, $10847$, $10891$, $10957$, $10979$, $11177$, $11243$, $11287$, $11353$, $11551$, $11617$, or when $p$ is a prime such that $p \equiv 1 \pmod {11}$ and $p > 11903$.
\end{thm}

In general, the arcs constructed here are not complete arcs.  Interestingly, calculations in $\GAP$ show that, for large enough primes $p$ such that $p \equiv 1 \mod 33$, the union of a $\PSL(2,11)$-transitive $110$-arc and a $\PSL(2,11)$-transitive $60$-arc is actually a $170$-arc in $\PG(4,p)$; see Calculation \ref{calc:170}.  On the other hand, while it is clear that the arcs constructed here are not complete for sufficiently large values of $p$, it is unclear whether or not the arcs are complete in some of the ``small'' spaces.  However, the smallest space in which an arc is constructed in this paper is $\PG(4,1277)$, which contains $2661361144381$ points, making the problem intractable computationally. 

This paper is organized as follows.  In Section \ref{sect:5dimrep}, we give the details of an irreducible $5$-dimensional representation of $\PSL(2,11)$ over the complex numbers.  In Section \ref{sect:verify}, we make explicit our method for verifying that a set of projective points is a $k$-arc.  We discuss the $110$-arcs in Section \ref{sect:110}, and we discuss the $60$-arcs in Section \ref{sect:60}.  Finally, since there are extensive calculations done in $\GAP$ \cite{GAP}, many with the aid of the package FinInG \cite{fining}, we detail the calculations done and programs used in Appendix \ref{sect:code}.

\section{A 5-dimensional irreducible representation of PSL(2,11)}
\label{sect:5dimrep}

From the Atlas \cite{atlas}, the group $\PSL(2,11)$ has a $5$-dimensional irreducible representation over $\C$, and generators for this representation in $\GAP$ \cite{GAP} can be found at \cite{onlineatlas}; see Remark \ref{rem:GAPPSL}.  Indeed, if $\xi$ is a primitive $11^{\text{th}}$ root of unity, $c := \xi + \xi^3 +  \xi^4 + \xi^5 + \xi^9 = (-1 + i\sqrt{11})/2$, and
\[ A := \begin{pmatrix}
0&1&0&0&0\\
1&0&0&0&0\\
0&0&0&0&1\\
1&-1&c&1&-c\\
0&0&1&0&0
\end{pmatrix}, 
B := \begin{pmatrix}
0&0&0&1&0\\
0&0&1&0&0\\
0&-1&-1&0&0\\
c+1&0&0&-1&c+2\\
1&0&0&-1&1
\end{pmatrix},\]
then $\langle A, B \rangle \cong \PSL(2,11)$.  This irreducible representation of $\PSL(2,11)$ may be viewed as a matrix representation where each matrix has coefficients in $\Z[c]$, where $c$ (as defined above) is a root of the irreducible polynomial $c^2 + c + 3$.  We could also view these as matrices in $\Z[\xi]$.  Indeed, by \cite[Table 8.19]{BHRD}, we see that $\PSL(2,11)$ is a maximal subgroup of $\PGL(5,q)$ whenever $q^5 \equiv 1 \pmod {11}$, $q \neq 3$.  (The group $\PGL(5,3)$ contains a maximal subgroup isomorphic to the Mathieu group $M_{11}$, which contains a subgroup isomorphic to $\PSL(2,11)$.)   Finally, we note that while there exists another $5$-dimensional irreducible representation of $\PSL(2,11)$ over $\Q[\xi]$, the outcomes of the calculations done later in this paper (e.g., for which primes a $k$-arc exists) are the same if this other representation of $\PSL(2,11)$ is chosen.  These additional calculations are not included in this paper for the sake of brevity.      

\section{Verifying that a set of projective points is an arc}
\label{sect:verify}

The purpose of this section is to provide a computationally effective method for determining whether the reduction modulo $p$ of a set of points $\calP$ contained in $n$-dimensional projective space over a number field $\calQ(\xi)$ to $\PG(n,q)$, where $q$ is some power of $p$, is an arc.

The lemma that follows gives a sufficient condition for a set of points to be a $k$-arc.  This condition is well-known, and its proof is omitted here.  

\begin{defi}
Given a field $\mathbb{K}$, a subset $J$ of size five of $\mathbb{K}^5$, and an ordering $\sigma = (v_1, v_2, v_3, v_4, v_5)$ of the elements of $J$, define $X_{J,\sigma}$ to be the matrix whose $i^{th}$ row is $v_i$, and define $[J]$ to be the set of all matrices $X_{J, \sigma}$ for a fixed subset $J$ of $\mathbb{K}^5$.  If $K := \{P_1, \dots, P_5\}$ are projective points in $\PG(4,\mathbb{K})$, define $[K]$ to be the set of all $X_{J, \sigma}$, where $P_i$ is associated to the linear subspace $\langle v_i \rangle$ of $\mathbb{K}^5$ and $J:= \{v_1, v_2, v_3, v_4, v_5\}$ with an ordering $\sigma = (v_1, v_2, v_3, v_4, v_5)$.
\end{defi}

\begin{lem}
\label{lem:verify}
 If $\calA$ is a set of $k$ projective points in $\PG(4,\bbK)$, where $\bbK$ is a field, and, for all subsets $K$ of $\calA$ of size five, $\det(X) \neq 0$ for some $X \in [K]$, then $\calA$ is a $k$-arc of $\PG(4,\bbK)$.
\end{lem}

It is easy to see that an analogous result holds for $\PG(n,\bbK)$, where $n \neq 4$.  Moreover, Lemma \ref{lem:verify} provides us with the following method to find an arc in $\PG(4,q)$.  Assume that we find a $k$-arc $\calP$ in $\PG(4, \Q(\xi))$, where $\Q(\xi)$ is a number field.  Lemma \ref{lem:verify} implies that, given any subset $K$ of $\calP$ of size five, we know that $\det(X) \neq 0$ for some $X \in [K]$.  Roughly speaking, if each of these determinants is nonzero in the reduction modulo $p$, then the image of $\calP$ will be a $k$-arc in the reduction modulo $p$.  In fact, there can only be a finite number of primes $p$ for which the reduction of $\calP$ modulo $p$ fails to be a $k$-arc.  The remainder of this section is dedicated to first formalizing this idea (Proposition \ref{prop:verify}) and then finding an effective method for computing a finite list of primes for which there may fail to be a $k$-arc in the reduction (Lemma \ref{lem:badprimes}). 

Henceforth in this section, let $f(x)$ be an irreducible polynomial over $\Z$, and let $R := \Z[\zeta]$, where $\zeta$ is a root of $f$. By abuse of notation, we identify $R$ with the field $\Z[x]/\langle f(x) \rangle$.  Let $\phi_{f(x)}$ be the homomorphism
\[ \phi_{f(x)}: \Z[x] \rightarrow \Z[x]/\langle f(x) \rangle \cong R.\]
For a prime $p$ such that $f(x)$ is still irreducible modulo $p$, let $\phi_{p}$ be the homomorphism \[ \phi_{p}: R \cong \Z[x]/\langle f(x) \rangle  \rightarrow \Z/p\Z[x]/\langle f(x)\rangle.\] Thus
\[ \phi_{f(x)}\phi_p : \Z[x] \rightarrow \Z/p\Z[x]/\langle f(x)\rangle,\]
and, since $\Ker(\phi_{f(x)}\phi_p) = \langle p, f(x) \rangle$ is a maximal ideal of $\Z[x]$, the image of $R$ under $\phi_p$ is a field, and we let $\GF(q)$ be the finite field $R^{\phi_{p}}$.   

In the following proposition, note that $\calP$ corresponds to a $k$-arc in $\PG(4, \calQ(\zeta))$, and so what we are really doing is verifying that, for infinitely many primes $p$, the reduction of this $k$-arc modulo $p$ is still a $k$-arc in the associated finite projective space.

\begin{prop}
\label{prop:verify}
 Let $\calP$ be a subset of $R^5$ of size $k$, and suppose for each subset $J$ of $\calP$ of size five that $\det(X) \neq 0$ for some $X \in [J]$.  Let $\mathcal{Q}$ be an infinite set of prime powers such that, if $q \in \mathcal{Q}$, then $\GF(q) \cong R^{\phi_{p}}$ for some prime $p$.   Then, there exist infinitely many $q \in \mathcal{Q}$ such that $\calP_p := \calP^{\phi_{p}}$ is a $k$-arc in $\PG(4, q)$. 
\end{prop}

\begin{proof}
 Assume that $p$ is a prime such that $\calP_p$ is not a $k$-arc of $\PG(4, R^{\phi_{p}}) = \PG(4,q)$.  This implies that there is a set of five projective points $K = \{P_1, P_2, P_3, P_4, P_5\}$ of $\calP_p$ contained in a hyperplane, which implies that $\det(X) = 0$ for every matrix $X \in [K]$.  Let each $P_i$ correspond to the vector $v_i \in \calP$ and $J = \{v_1, v_2, v_3, v_4, v_5\}$.  Since $\det(Y) \neq 0$ for some $Y \in [J]$ by assumption and the matrices of $[J]$ differ by elementary row operations, this implies that no matrix in $[J]$ has zero determinant in $R$.  Hence, for some matrix $Y \in [J]$, $\det(Y)$ is nonzero in $R$, but the corresponding determinant is zero in $\GF(q)$. 
 
 On the other hand, there are only finitely many subsets of size five of $\calP$, namely ${k \choose 5}$.  For each subset $J$ of size five, there are only finitely many different matrices in $[J]$, which means there are only finitely many different values of $\det(X)$, where $X \in [J]$.  Since each such determinant is nonzero, for each $X \in [J]$, there are only finitely many different primes $p$ such that $\det(X^{\phi_{p}}) = 0$ in $\GF(q)$, where $q$ is a power of $p$.  Therefore, there are only finitely many primes such that $\calP_p$ is not a $k$-arc in $\PG(4,q)$.  The result follows. 
\end{proof}

More practically, we want to be able to determine in an efficient manner precisely which primes will yield a $k$-arc.  We first define notation that will be useful.

\begin{defi}
Let $\calP$ be a subset of $R^5$ of size $k$, and suppose for each subset $J$ of $\calP$ of size five that $\det(X) \neq 0$ for some $X \in [J]$.  Let $\mathcal{X}_\calP$ be a set of representatives of $[J]$ as $J$ runs over the subsets of size five of $\calP$. Let $\mathcal{D}_\calP := \{\det(X) : X \in \mathcal{X}_\calP\}$, let $\mathcal{E}_\calP$ be set of all minimal polynomials of the elements of $\mathcal{D}_\calP$ over $\Q$, and let $\mathcal{F}_\calP$ be the set of all polynomials in $\mathcal{E}_\calP$ rationalized such that all coefficients are in $\Z$ (and the coefficients of polynomials in $\mathcal{F}_\calP$ have greatest common divisor $1$). We define $\mathcal{C}_\calP$ to be the set of prime factors of the constant terms of the polynomials in $\mathcal{F}_\calP$. 
\end{defi}

The following lemma provides a method for finding every prime such that $\mathcal{P}^{\phi_p}$ is not a $k$-arc.  

\begin{lem}
\label{lem:badprimes}
 If $p$ is a prime such that $f(x)$ is irreducible modulo $p$ and $p \not\in \mathcal{C}_\calP$, then $\calP_p := \calP^{\phi_{p}}$ is a $k$-arc in $\PG(4, q)$, where $\GF(q) \cong R^{\phi_{p}}$.   
\end{lem}

\begin{proof}
We first make an observation.  Fix $J \subset \calP$, $|J| = 5$, and let $X$, $Y \in [J]$.  Since $X$ and $Y$ differ only in the order in which the rows appear in each matrix, $\det(Y) = \pm \det(X)$, and so $\det(X)^{\phi_{p}} = 0$ if and only if $\det(Y)^{\phi_{p}} = 0$.  It thus suffices to consider a representative of each class $[J]$.

Assume that $p$ is a prime but that $\calP_p := \calP^{\phi_{p}}$ is not a $k$-arc in $\PG(4, q)$, where $\GF(q) \cong R^{\phi_{p}}$.  This means that there exists some subset $J$ of $\calP$ such that $\det(X)^{\phi_{p}} = 0$ for all $X \in [J]$.  Let $d := \det(X)$ for some $X \in [J]$, and define $f_d(x)$ to be the minimal polynomial of $d$ over $\Q$.  There exists $N \in \N$ such that $g_d(x) := N \cdot f_d(x) \in \Z[x]$ and the coefficients of $g_d(x)$ have greatest common divisor $1$.  Since $f_d(d) = 0$ in $\Q$, $g_d(d)^{\phi_p} = 0$ in $\GF(q)$.  On the other hand $d^{\phi_{p}} = 0$ in $\GF(q)$, and so
\[g_d(d)^{\phi_{p}} = c_d^{\phi_{p}},\] where $c_d$ is the constant term of $g_d(x)$, which implies that $p \mid c_d$, i.e., $p \in \mathcal{C}_\calP$.  The result follows.
\end{proof}

Indeed, in practice the hypotheses of Lemma \ref{lem:badprimes} can be checked relatively quickly in $\GAP$ for $k$-arcs of moderate size.

\section{Transitive PSL(2,11)-invariant 110-arcs in PG(4,q)}
\label{sect:110}
\subsection{Existence}

We recall from Section \ref{sect:5dimrep} that there is a $5$-dimensional representation of $\PSL(2,11)$ over $\Z[\xi]$ with generators 
\[ A := \begin{pmatrix}
0&1&0&0&0\\
1&0&0&0&0\\
0&0&0&0&1\\
1&-1&c&1&-c\\
0&0&1&0&0
\end{pmatrix}, 
B := \begin{pmatrix}
0&0&0&1&0\\
0&0&1&0&0\\
0&-1&-1&0&0\\
c+1&0&0&-1&c+2\\
1&0&0&-1&1
\end{pmatrix},\]
where $c = \xi + \xi^3 +  \xi^4 + \xi^5 + \xi^9 = (-1 + i\sqrt{11})/2$ for a primitive $11^\text{th}$ root of unity $\xi$. 

By \cite{onlineatlas}, $M := ABABABABB$ has order $6$ in $G:= \langle A, B \rangle \cong \PSL(2,11)$.  Suppose $v$ is an eigenvector of $M$.  This means $vM \in \langle v \rangle$, and so $\langle v \rangle$ is fixed by $H:= \langle M \rangle$.  If $T_1, \dots T_{110}$ are representatives of the distinct cosets of $\langle M \rangle$, then we will see that the $110$-arc we are looking for comes from the projective points associated to the linear subspaces $\langle vT_1 \rangle, \dots, \langle vT_{110}\rangle$.  Note that, in order to get a suitable eigenvector $v$, there need eventually to exist third roots of unity in the finite field, and so we work in $\Z[\xi_{33}]$, where $\xi_{33}$ is a primitive $33^\text{rd}$ root of unity.

\begin{prop}
\label{prop:110set}
Let $G = \langle A, B \rangle$ and $H: = \langle M \rangle$, as above.  Let $T_1 = I, \dots, T_{110}$ be a set of representatives of the distinct cosets of $H$ in $G$.  Then, there exists an eigenvector $v$ of $H$ such that any subset $J$ of size five of the set $\calP := \{v, vT_2, \dots, vT_{110}\}$ has the property that $\det(X) \neq 0$ for some $X \in [J]$. 
\end{prop}

\begin{proof}
 This follows by Calculation \ref{calc:110}.  We note that it took approximately 6.8 GB of RAM to complete this calculation.
\end{proof}

\begin{thm}
 There are infinitely many primes $p$ such that, if $\GF(q) = \Z[\xi_{33}]^{\phi_{p}}$, then there exists a $\PSL(2,11)$-transitive $110$-arc in $\PG(4,q)$.
\end{thm}

\begin{proof}
 This follows from Proposition \ref{prop:110set} and Proposition \ref{prop:verify}.
\end{proof}

\subsection{Examples}

With the aid of $\GAP$ and the package FinInG \cite{fining}, we are able to calculate specific primes and prime powers $q$ such that there is a $\PSL(2,11)$-transitive $110$-arc in $\PG(4,q)$.  

Assume that $p$ is a prime and that there exists a faithful representation of $\PSL(2,11)$ over $\GF(p)$.  In this case, we need both primitive $11^\text{th}$ and primitive $3^\text{rd}$ roots of unity to exist in $\GF(p)$; hence, we assume that $p^5 \equiv 1 \pmod {11}$ and $p \equiv 1 \pmod 3$.  We are now ready to prove Theorem \ref{thm:110prime}.

\begin{proof}[Proof of Theorem \ref{thm:110prime}]
 That there are infinitely many such primes follows from the fact that $\GF(p) = \Z[\xi_{33}]^{\phi_{p}}$ whenever $p \equiv 1 \pmod {33}$ by the above discussion, Proposition \ref{prop:110set}, and Proposition \ref{prop:verify}.  
 
 By Lemma \ref{lem:badprimes} and Calculation \ref{calc:goodprimes110}, the possible primes $p$ for which $\PG(4,p)$ does not contain a $110$-arc are explicitly known.  These calculations have further been verified for all primes $p$ such that $p^5 \equiv 1 \pmod {33}$ and $p<65000$ using Functions \ref{func:110arcp} and \ref{func:110arcsp}.  (Function \ref{func:110arcp} is called by Function \ref{func:110arcsp}, and Function \ref{func:span} is called by Function \ref{func:110arcp}.)  Moreover, we have verified all of these calculations for both irreducible $5$-dimensional representations of $\PSL(2,11)$ over $\C$, although these calculations are omitted since they are analogous to the ones listed.
\end{proof} 

There are many prime powers such that $\GF(p)$ contains a primitive $11^\text{th}$ root of unity but not a primitive $3^\text{rd}$ root of unity.  In this case, $p^5 \equiv 1 \pmod {11}$ and $p \equiv 2 \pmod 3$.  In this case, we need to adjoin a primitive third root of unity to $\GF(p)$, and so we examine $\GF(p^2)$.  We are now ready to prove Theorem \ref{thm:110prime2}.

\begin{proof}[Proof of Theorem \ref{thm:110prime2}]
 That there are infinitely many such primes follows from the fact that $\GF(p^2) = \Z[\xi_{33}]^{\phi_{p}}$ whenever $p^5 \equiv 23 \pmod {11}$ and $p$ by the above discussion, Proposition \ref{prop:110set}, and Proposition \ref{prop:verify}. 
 
 By Lemma \ref{lem:badprimes} and Calculation \ref{calc:goodprimes110}, the possible primes $p$ for which $\PG(4,p^2)$ does not contain a $110$-arc are explicitly known.  These calculations have further been verified for all primes $p$ such that $p^5 \equiv 23 \pmod {33}$ and $p<65000$ using Functions \ref{func:110arc2} and \ref{func:110arcs2}.  (Function \ref{func:110arc2} is called by Function \ref{func:110arcs2}, and Function \ref{func:span} is called by Function \ref{func:110arc2}.)  Moreover, we have verified all of these calculations for both irreducible $5$-dimensional representations of $\PSL(2,11)$ over $\C$, although these are omitted since they are analogous to the ones listed.
\end{proof}

\section{Transitive PSL(2,11)-invariant 60-arcs in PG(4,q)}
\label{sect:60}
\subsection{Existence}

We recall from Section \ref{sect:5dimrep} that there is a $5$-dimensional representation of $\PSL(2,11)$ over $\Z[\xi]$ with generators 
\[ A := \begin{pmatrix}
0&1&0&0&0\\
1&0&0&0&0\\
0&0&0&0&1\\
1&-1&c&1&-c\\
0&0&1&0&0
\end{pmatrix}, 
B := \begin{pmatrix}
0&0&0&1&0\\
0&0&1&0&0\\
0&-1&-1&0&0\\
c+1&0&0&-1&c+2\\
1&0&0&-1&1
\end{pmatrix},\]
where $c = \xi + \xi^3 +  \xi^4 + \xi^5 + \xi^9 = (-1 + i\sqrt{11})/2$ for a primitive $11^\text{th}$ root of unity $\xi$. 

By \cite{onlineatlas}, $M:= AB$ is an element of order $11$ in $\langle A,B \rangle \cong \PSL(2,11)$. Suppose $v$ is an eigenvector of $M$.  This means $vM \in \langle v \rangle$, and so $\langle v \rangle$ is fixed by $H:= \langle M \rangle$.  If $T_1, \dots T_{60}$ are representatives of the distinct cosets of $\langle M \rangle$, then we will see that the $60$-arc we are looking for comes from the projective points associated to the linear subspaces $\langle vT_1 \rangle, \dots, \langle vT_{60}\rangle$.  

\begin{prop}
\label{prop:60set}
Let $G = \langle A, B \rangle$ and $H: = \langle M \rangle$, as above.  Let $T_1 = I, \dots, T_{60}$ be a set of representatives of the distinct cosets of $H$ in $G$.  Then, there exists an eigenvector $v$ of $H$ such that any subset $J$ of size five of the set $\calP := \{v, vT_2, \dots, vT_{60}\}$ has the property that $\det(X) \neq 0$ for some $X \in [J]$. 
\end{prop}

\begin{proof}
 This follows by Calculation \ref{calc:60}.
\end{proof}

\begin{thm}
 There are infinitely many primes $p$ such that, if $\GF(q) = \Z[\xi]^{\phi_{p}}$, then there exists a $\PSL(2,11)$-transitive $60$-arc in $\PG(4,q)$.
\end{thm}

\begin{proof}
 This follows from Proposition \ref{prop:60set} and Proposition \ref{prop:verify}.
\end{proof}

\subsection{Examples}
With the aid of $\GAP$ and the package FinInG \cite{fining}, we are able to calculate specific primes and prime powers $q$ such that there is a $\PSL(2,11)$-transitive $60$-arc in $\PG(4,q)$.  

Assume that $p$ is a prime and that there exists a faithful representation of $\PSL(2,11)$ over $\GF(p)$.  In this case, we need the constant $c$ to exist in $\GF(p)$; hence, we assume that $p^5 \equiv 1 \pmod {11}$.  Moreover, the eigenvalue we choose is actually a primitive $11^\text{th}$ root of unity, so we assume further that $p \equiv 1 \pmod {11}$.  We are now ready to prove Theorem \ref{thm:60prime}.    

\begin{proof}[Proof of Theorem \ref{thm:60prime}]
 That there are infinitely many such primes follows from the fact that $\GF(p) = \Z[\xi_{11}]^{\phi_{p}}$ whenever $p \equiv 1 \pmod {11}$ by the above discussion, Proposition \ref{prop:60set}, and Proposition \ref{prop:verify}.  
 
 By Lemma \ref{lem:badprimes} and Calculation \ref{calc:goodprimes60}, the possible primes $p$ for which $\PG(4,p)$ does not contain a $60$-arc are explicitly known.  These calculations have further been verified for all primes $p$ such that $p^5 \equiv 1 \pmod {11}$ and $p<65000$ using Functions \ref{func:60arc} and \ref{func:60arcs}.  (Function \ref{func:60arc} is called by Function \ref{func:60arcs}, and Function \ref{func:span} is called by Function \ref{func:60arc}.)  Moreover, we have verified all of these calculations for both irreducible $5$-dimensional representations of $\PSL(2,11)$ over $\C$, although these are omitted since they are analogous to the ones listed.
\end{proof}

\noindent\textsc{Acknowledgements.}  The authors would like to thank John Bamberg for helpful discussions about the $\GAP$ package FinInG and Jan De Beule for $\GAP$ code that allows for calculations using FinInG beyond what is normally possible in $\GAP$.  The authors would also like to thank the anonymous referees for many useful suggestions that improved the readability of this paper.

\appendix
\section{GAP code and calculations}
\label{sect:code}
We include in this section the $\GAP$ code used in this paper.

\begin{rem}
 \label{rem:GAPPSL}
The following code, taken directly from \cite{onlineatlas}, can be used to generate one of the $5$-dimensional complex irreducible representations of $\PSL(2,11)$. 
 \begin{verbatim}
# Character: X2
# Comment: perm rep on 55a pts
# Ind: 0
# Ring: C
# Sparsity: 62%
# Checker result: pass
# Conjugacy class representative result: pass
local a, A, b, B, c, C, w, W, i, result, delta, idmat;
result := rec();
w := E(3); W := E(3)^2;
a := E(5)+E(5)^4; A := -1-a; # b5, b5*
b := E(7)+E(7)^2+E(7)^4; B := -1-b;  # b7, b7**
c := E(11)+E(11)^3+E(11)^4+E(11)^5+E(11)^9; C := -1-c; # b11, b11**
i := E(4);
result.comment := "L211 as 5 x 5 matrices\n";
result.generators := [
[[0,1,0,0,0],
[1,0,0,0,0],
[0,0,0,0,1],
[1,-1,c,1,-c],
[0,0,1,0,0]]
,
[[0,0,0,1,0],
[0,0,1,0,0],
[0,-1,-1,0,0],
[-C,0,0,-1,-c-2*C],
[1,0,0,-1,1]]];
return result;
 \end{verbatim}

\end{rem}

\begin{calc}
 \label{calc:110}
 The following calculation is used for the proof of Proposition \ref{prop:110set}.  This calculation took approximately 6.8 GB of RAM to complete.
 \begin{verbatim}
gap> c := E(11)+E(11)^3+E(11)^4+E(11)^5+E(11)^9; C := -1-c;
E(11)+E(11)^3+E(11)^4+E(11)^5+E(11)^9
E(11)^2+E(11)^6+E(11)^7+E(11)^8+E(11)^10
gap> A:= [[0,1,0,0,0],
> [1,0,0,0,0],
> [0,0,0,0,1],
> [1,-1,c,1,-c],
> [0,0,1,0,0]];
[ [ 0, 1, 0, 0, 0 ], [ 1, 0, 0, 0, 0 ], [ 0, 0, 0, 0, 1 ], 
  [ 1, -1, E(11)+E(11)^3+E(11)^4+E(11)^5+E(11)^9, 1, 
      -E(11)-E(11)^3-E(11)^4-E(11)^5-E(11)^9 ], [ 0, 0, 1, 0, 0 ] ]
gap> B:= [[0,0,0,1,0],
> [0,0,1,0,0],
> [0,-1,-1,0,0],
> [-C,0,0,-1,-c-2*C],
> [1,0,0,-1,1]];
[ [ 0, 0, 0, 1, 0 ], [ 0, 0, 1, 0, 0 ], [ 0, -1, -1, 0, 0 ], 
  [ -E(11)^2-E(11)^6-E(11)^7-E(11)^8-E(11)^10, 0, 0, -1, 
      -E(11)-2*E(11)^2-E(11)^3-E(11)^4-E(11)^5-2*E(11)^6-2*E(11)^7-2*E(11)^8
         -E(11)^9-2*E(11)^10 ], [ 1, 0, 0, -1, 1 ] ]
gap> G:= Group(A,B);
<matrix group with 2 generators>
gap> Order(G);
660
gap> M:= A*B*A*B*A*B*A*B*B;
[ [ -1, 2, -2*E(11)-E(11)^2-2*E(11)^3-2*E(11)^4-2*E(11)^5-E(11)^6-E(11)^7
         -E(11)^8-2*E(11)^9-E(11)^10, -1, 
      E(11)+E(11)^3+E(11)^4+E(11)^5+E(11)^9 ], 
  [ 0, -2*E(11)-E(11)^2-2*E(11)^3-2*E(11)^4-2*E(11)^5-E(11)^6-E(11)^7-E(11)^8
         -2*E(11)^9-E(11)^10, E(11)^2+E(11)^6+E(11)^7+E(11)^8+E(11)^10, 
      E(11)+E(11)^3+E(11)^4+E(11)^5+E(11)^9, 2 ], 
  [ 1, -2, E(11)+E(11)^3+E(11)^4+E(11)^5+E(11)^9, 1, 
      E(11)^2+E(11)^6+E(11)^7+E(11)^8+E(11)^10 ], 
  [ -1, E(11)+E(11)^3+E(11)^4+E(11)^5+E(11)^9, 
      -E(11)-2*E(11)^2-E(11)^3-E(11)^4-E(11)^5-2*E(11)^6-2*E(11)^7-2*E(11)^8
         -E(11)^9-2*E(11)^10, E(11)^2+E(11)^6+E(11)^7+E(11)^8+E(11)^10, 
      3*E(11)+2*E(11)^2+3*E(11)^3+3*E(11)^4+3*E(11)^5+2*E(11)^6+2*E(11)^7
         +2*E(11)^8+3*E(11)^9+2*E(11)^10 ], 
  [ -E(11)^2-E(11)^6-E(11)^7-E(11)^8-E(11)^10, 0, 0, -1, 
      -E(11)-2*E(11)^2-E(11)^3-E(11)^4-E(11)^5-2*E(11)^6-2*E(11)^7-2*E(11)^8
         -E(11)^9-2*E(11)^10 ] ]
gap> Order(M);
6
gap> F:= Field(E(33));
CF(33)
gap> eigs:= Eigenvectors(F,M);
[ [ 1, 1/3*E(11)^2+1/3*E(11)^6+1/3*E(11)^7+1/3*E(11)^8+1/3*E(11)^10, 0, 
      -1/3*E(11)^2-1/3*E(11)^6-1/3*E(11)^7-1/3*E(11)^8-1/3*E(11)^10, 
      -E(11)-1/3*E(11)^2-E(11)^3-E(11)^4-E(11)^5-1/3*E(11)^6-1/3*E(11)^7
         -1/3*E(11)^8-E(11)^9-1/3*E(11)^10 ], 
  [ 1, -5/3*E(33)^2-E(33)^5-2/3*E(33)^7-5/3*E(33)^8-2/3*E(33)^10-2/3*E(33)^13
         -E(33)^14-5/3*E(33)^17-2/3*E(33)^19-E(33)^20-E(33)^23-E(33)^26
         -2/3*E(33)^28-5/3*E(33)^29-5/3*E(33)^32, 
      -E(33)-5/3*E(33)^2-E(33)^4-2*E(33)^5-4/3*E(33)^7-5/3*E(33)^8
         -4/3*E(33)^10-4/3*E(33)^13-2*E(33)^14-E(33)^16-5/3*E(33)^17
         -4/3*E(33)^19-2*E(33)^20-2*E(33)^23-E(33)^25-2*E(33)^26-4/3*E(33)^28
         -5/3*E(33)^29-E(33)^31-5/3*E(33)^32, 
      4/3*E(33)^2+E(33)^5+E(33)^7+4/3*E(33)^8+E(33)^10+E(33)^13+E(33)^14
         +4/3*E(33)^17+E(33)^19+E(33)^20+E(33)^23+E(33)^26+E(33)^28
         +4/3*E(33)^29+4/3*E(33)^32, 
      2*E(33)+2*E(33)^4+E(33)^5+4/3*E(33)^7+4/3*E(33)^10+4/3*E(33)^13+E(33)^14
         +2*E(33)^16+4/3*E(33)^19+E(33)^20+E(33)^23+2*E(33)^25+E(33)^26
         +4/3*E(33)^28+2*E(33)^31 ], 
  [ 1, -E(33)-2/3*E(33)^2-E(33)^4-5/3*E(33)^7-2/3*E(33)^8-5/3*E(33)^10
         -5/3*E(33)^13-E(33)^16-2/3*E(33)^17-5/3*E(33)^19-E(33)^25
         -5/3*E(33)^28-2/3*E(33)^29-E(33)^31-2/3*E(33)^32, 
      -2*E(33)-4/3*E(33)^2-2*E(33)^4-E(33)^5-5/3*E(33)^7-4/3*E(33)^8
         -5/3*E(33)^10-5/3*E(33)^13-E(33)^14-2*E(33)^16-4/3*E(33)^17
         -5/3*E(33)^19-E(33)^20-E(33)^23-2*E(33)^25-E(33)^26-5/3*E(33)^28
         -4/3*E(33)^29-2*E(33)^31-4/3*E(33)^32, 
      E(33)+E(33)^2+E(33)^4+4/3*E(33)^7+E(33)^8+4/3*E(33)^10+4/3*E(33)^13
         +E(33)^16+E(33)^17+4/3*E(33)^19+E(33)^25+4/3*E(33)^28+E(33)^29
         +E(33)^31+E(33)^32, 
      E(33)+4/3*E(33)^2+E(33)^4+2*E(33)^5+4/3*E(33)^8+2*E(33)^14+E(33)^16
         +4/3*E(33)^17+2*E(33)^20+2*E(33)^23+E(33)^25+2*E(33)^26+4/3*E(33)^29
         +E(33)^31+4/3*E(33)^32 ], 
  [ 1, 5/3*E(33)+5/3*E(33)^4+4/3*E(33)^5+E(33)^7+E(33)^10+E(33)^13
         +4/3*E(33)^14+5/3*E(33)^16+E(33)^19+4/3*E(33)^20+4/3*E(33)^23
         +5/3*E(33)^25+4/3*E(33)^26+E(33)^28+5/3*E(33)^31, 
      -2*E(33)^2-E(33)^5-E(33)^7-2*E(33)^8-E(33)^10-E(33)^13-E(33)^14
         -2*E(33)^17-E(33)^19-E(33)^20-E(33)^23-E(33)^26-E(33)^28-2*E(33)^29
         -2*E(33)^32, -E(33)+E(33)^2-E(33)^4+E(33)^8-E(33)^16+E(33)^17
         -E(33)^25+E(33)^29-E(33)^31+E(33)^32, 
      7/3*E(33)+2*E(33)^2+7/3*E(33)^4+8/3*E(33)^5+2*E(33)^7+2*E(33)^8
         +2*E(33)^10+2*E(33)^13+8/3*E(33)^14+7/3*E(33)^16+2*E(33)^17
         +2*E(33)^19+8/3*E(33)^20+8/3*E(33)^23+7/3*E(33)^25+8/3*E(33)^26
         +2*E(33)^28+2*E(33)^29+7/3*E(33)^31+2*E(33)^32 ], 
  [ 1, 4/3*E(33)+E(33)^2+4/3*E(33)^4+5/3*E(33)^5+E(33)^8+5/3*E(33)^14
         +4/3*E(33)^16+E(33)^17+5/3*E(33)^20+5/3*E(33)^23+4/3*E(33)^25
         +5/3*E(33)^26+E(33)^29+4/3*E(33)^31+E(33)^32, 
      -E(33)-E(33)^2-E(33)^4-2*E(33)^7-E(33)^8-2*E(33)^10-2*E(33)^13-E(33)^16
         -E(33)^17-2*E(33)^19-E(33)^25-2*E(33)^28-E(33)^29-E(33)^31-E(33)^32, 
      -E(33)^5+E(33)^7+E(33)^10+E(33)^13-E(33)^14+E(33)^19-E(33)^20-E(33)^23
         -E(33)^26+E(33)^28, 
      8/3*E(33)+2*E(33)^2+8/3*E(33)^4+7/3*E(33)^5+2*E(33)^7+2*E(33)^8
         +2*E(33)^10+2*E(33)^13+7/3*E(33)^14+8/3*E(33)^16+2*E(33)^17
         +2*E(33)^19+7/3*E(33)^20+7/3*E(33)^23+8/3*E(33)^25+7/3*E(33)^26
         +2*E(33)^28+2*E(33)^29+8/3*E(33)^31+2*E(33)^32 ] ]
gap> Length(eigs);
5
gap> eigvals:= Eigenvalues(F,M);
[ 1, E(3), E(3)^2, -E(3)^2, -E(3) ]
gap> eigs[2]*M = eigvals[2]*eigs[2];
true
gap> H:= Group(M);
<matrix group with 1 generators>
gap> Order(H);
6
gap> reps:= List(RightCosets(G,H), Representative);;
gap> Length(reps);
110
gap> set:= List(reps, i -> eigs[2]*i);;
gap> set[1] = eigs[2];
true
gap> preset:= List([2..110], i -> set[i]);;
gap> precomb:= Combinations(preset, 4);;

#Since PSL(2,11) is transitive on the set of subspaces generated by the 
#vectors in "set," if there is a linear dependence among any five vectors,
#then there is a linear dependence among some set of five vectors that
#includes the first element of "set."  This allows us to save memory and
#time in the calculation.

gap> Length(precomb);
5563251
gap> sets:= List(precomb, i -> Concatenation([set[1]], i));;
gap> dets110:= List(sets, i -> DeterminantMat(i));;
gap> Length(dets110);
5563251
gap> 0 in dets110;
false
 \end{verbatim}
\end{calc}

\begin{calc}
 \label{calc:60}
 The following calculation is used for the proof of Proposition \ref{prop:60set}.
 \begin{verbatim}
gap> c := E(11)+E(11)^3+E(11)^4+E(11)^5+E(11)^9; C := -1-c;
E(11)+E(11)^3+E(11)^4+E(11)^5+E(11)^9
E(11)^2+E(11)^6+E(11)^7+E(11)^8+E(11)^10
gap> A:= [[0,1,0,0,0],
> [1,0,0,0,0],
> [0,0,0,0,1],
> [1,-1,c,1,-c],
> [0,0,1,0,0]];
[ [ 0, 1, 0, 0, 0 ], [ 1, 0, 0, 0, 0 ], [ 0, 0, 0, 0, 1 ], 
  [ 1, -1, E(11)+E(11)^3+E(11)^4+E(11)^5+E(11)^9, 1, 
      -E(11)-E(11)^3-E(11)^4-E(11)^5-E(11)^9 ], [ 0, 0, 1, 0, 0 ] ]
gap> B:= [[0,0,0,1,0],
> [0,0,1,0,0],
> [0,-1,-1,0,0],
> [-C,0,0,-1,-c-2*C],
> [1,0,0,-1,1]];
[ [ 0, 0, 0, 1, 0 ], [ 0, 0, 1, 0, 0 ], [ 0, -1, -1, 0, 0 ], 
  [ -E(11)^2-E(11)^6-E(11)^7-E(11)^8-E(11)^10, 0, 0, -1, 
      -E(11)-2*E(11)^2-E(11)^3-E(11)^4-E(11)^5-2*E(11)^6-2*E(11)^7-2*E(11)^8
         -E(11)^9-2*E(11)^10 ], [ 1, 0, 0, -1, 1 ] ]
gap> G:= Group(A,B);
<matrix group with 2 generators>
gap> Order(G);
660
gap> M:= A*B;
[ [ 0, 0, 1, 0, 0 ], [ 0, 0, 0, 1, 0 ], [ 1, 0, 0, -1, 1 ], 
  [ 1, -E(11)-E(11)^3-E(11)^4-E(11)^5-E(11)^9, 
      E(11)^2+E(11)^6+E(11)^7+E(11)^8+E(11)^10, 
      E(11)+E(11)^3+E(11)^4+E(11)^5+E(11)^9, 2 ], [ 0, -1, -1, 0, 0 ] ]
gap> Order(M);
11
gap> F:= Field(E(11));
CF(11)
gap> eigs:= Eigenvectors(F,M);
[ [ 1, E(11)^2-E(11)^9, E(11)+E(11)^2+E(11)^10, -E(11)^2-E(11)^10, 
      -2*E(11)-E(11)^2-E(11)^3-E(11)^4-E(11)^5-E(11)^6-E(11)^7-E(11)^8
         -2*E(11)^9-E(11)^10 ], 
  [ 1, -E(11)^5+E(11)^6, E(11)^3+E(11)^6+E(11)^8, -E(11)^6-E(11)^8, 
      -E(11)-E(11)^2-2*E(11)^3-E(11)^4-2*E(11)^5-E(11)^6-E(11)^7-E(11)^8
         -E(11)^9-E(11)^10 ], 
  [ 1, -E(11)^3+E(11)^8, E(11)^4+E(11)^7+E(11)^8, -E(11)^7-E(11)^8, 
      -E(11)-E(11)^2-2*E(11)^3-2*E(11)^4-E(11)^5-E(11)^6-E(11)^7-E(11)^8
         -E(11)^9-E(11)^10 ], 
  [ 1, -E(11)+E(11)^10, E(11)^5+E(11)^6+E(11)^10, -E(11)^6-E(11)^10, 
      -2*E(11)-E(11)^2-E(11)^3-E(11)^4-2*E(11)^5-E(11)^6-E(11)^7-E(11)^8
         -E(11)^9-E(11)^10 ], 
  [ 1, -E(11)^4+E(11)^7, E(11)^2+E(11)^7+E(11)^9, -E(11)^2-E(11)^7, 
      -E(11)-E(11)^2-E(11)^3-2*E(11)^4-E(11)^5-E(11)^6-E(11)^7-E(11)^8
         -2*E(11)^9-E(11)^10 ] ]
gap> Length(eigs);
5
gap> eigvals:= Eigenvalues(F,M);
[ E(11), E(11)^3, E(11)^4, E(11)^5, E(11)^9 ]
gap> eigs[1]*M = eigvals[1]*eigs[1];
true
gap> H:= Group(M);
<matrix group with 1 generators>
gap> Order(H);
11
gap> reps:= List(RightCosets(G,H), Representative);;
gap> Length(reps);
60
gap> set:= List(reps, i -> eigs[1]*i);;
gap> set[1] = eigs[1];
true
gap> preset:= List([2..60], i -> set[i]);;
gap> precomb:= Combinations(preset,4);;
gap> Length(precomb);
455126
gap> sets:= List(precomb, i -> Concatenation([set[1]], i));;
gap> dets60:= List(sets, i -> DeterminantMat(i));;
gap> Length(dets60);
455126
gap> 0 in dets60;
false  
 \end{verbatim}
\end{calc}
 
\begin{calc}
 \label{calc:goodprimes110}
The following calculation in $\GAP$ demonstrates the primes $p$ for which there may not exist a $\PSL(2,11)$-transitive $110$-arc in $\PG(4,p)$ or $\PG(4,p^2)$. 
 \begin{verbatim}
gap> const:= [];; Q:= Field(1);; for r in dets110 do 
pref:= MinimalPolynomial(Rationals,r); 
f:= DenominatorOfRationalFunction(pref)*pref; 
Add(const, CoefficientsOfUnivariatePolynomial(f)[1]); od;
gap> badprimes110:= [];; for c in const do pd:= PrimeDivisors(c); 
for p in pd do if not p in badprimes110 then Add(badprimes110, p); fi; od; od;
gap> Length(badprimes110);
2793
gap> max:= Maximum(badprimes110);
5373427
gap> badprimes1:= Filtered(badprimes110, p -> (p^5 mod 33 = 1));;    
gap> badprimes2:= Filtered(badprimes110, p -> (p^5 mod 33 = 23));;
gap> Length(badprimes1);
2728
gap> Length(badprimes2);
25
gap> Maximum(badprimes2);
3917
gap> badprimes2;
[ 5, 53, 23, 59, 47, 89, 191, 71, 113, 137, 269, 863, 251, 383, 179, 797, 
  929, 683, 3917, 353, 1871, 647, 257, 1439, 971 ]
  \end{verbatim}

\end{calc} 
 
\begin{calc}
 \label{calc:goodprimes60}
The following calculation in $\GAP$ demonstrates exactly which primes $p$ are such that there exists a $\PSL(2,11)$-transitive $60$-arc in $\PG(4,p)$. 
 \begin{verbatim}
gap> const:= [];; Q:= Field(1);; for r in dets60 do 
pref:= MinimalPolynomial(Rationals,r); 
f:= DenominatorOfRationalFunction(pref)*pref; 
Add(const, CoefficientsOfUnivariatePolynomial(f)[1]); od;
gap> badprimes60:= [];; for c in const do pd:= PrimeDivisors(c); 
for p in pd do if not p in badprimes60 then Add(badprimes60, p); fi; od; od;
gap> Length(badprimes60);
57
gap> badprimes60;
[ 11, 67, 23, 3, 419, 89, 331, 397, 727, 2179, 2377, 353, 199, 2663, 
2267, 1013, 4049, 881, 2927, 1123, 6491, 617, 2113, 1409, 1607, 463, 
859, 683, 661, 2, 1321, 947, 3719, 2531, 1871, 991, 2311, 3037, 3191, 
2003, 3851, 43, 3917, 3389, 3323, 3433, 2729, 3697, 5, 3257, 7591, 
11903, 1453, 4093, 3301, 7877, 109 ]
gap> max:= Maximum(badprimes60);
11903
gap> possprimes:= Filtered([3..11903], p -> IsPrime(p) and (p mod 11 = 1));;
gap> Length(possprimes);
141
gap> goodprimes:= Filtered(possprimes, p -> not p in badprimes60);
[ 1277, 1783, 2069, 2333, 2399, 2707, 2861, 2971, 3169, 3499, 3631, 
4027, 4159, 4357, 4423, 4621, 4643, 4951, 4973, 5039, 5171, 5237, 5281, 
5303, 5347, 5413, 5479, 5501, 5743, 5897, 6007, 6029, 6073, 6271, 6337, 
6359, 6469, 6689, 6733, 6997, 7019, 7129, 7151, 7283, 7349, 7393, 7459, 
7481, 7547, 7723, 7789, 8009, 8053, 8273, 8317, 8537, 8581, 8647, 8669, 
8713, 8779, 8867, 8933, 8999, 9043, 9109, 9241, 9439, 9461, 9769, 9791, 
9857, 9901, 9923, 9967, 10099, 10253, 10429, 10627, 10781, 10847, 10891, 
10957, 10979, 11177, 11243, 11287, 11353, 11551, 11617 ]
  \end{verbatim}

\end{calc}

\begin{func}
\label{func:span}
The following function is used in subsequent functions to calculate the projective dimension of a space that is spanned by a few projective points.
\begin{verbatim}
SpanSizes:= function(orbit, size)

local neworblist, l, comb, sets, spans, spansizes, point;

point:= orbit[1];
l:= Length(orbit);
neworblist:= List([2..l], i -> orbit[i]);;
comb:= Combinations(neworblist, size-1);;
sets:= List(comb, i -> Concatenation([point], i));;
spans:= List(sets, Span);;
spansizes:= List(spans, ProjectiveDimension);;
return [Collected(spansizes)];
end;
\end{verbatim}
\end{func}

\begin{func} 
\label{func:110arcp}
The following function checks whether $\PG(4,p)$ contains a $\PSL(2,11)$-transitive $110$-arc, where $p$ is a prime.
 \begin{verbatim}
110arcp:= function(p)

local l, cand, c, C, m1, m2, mat, eigs, test, pos, PS, P, v, H, o, 
list, i, vects, orbs;

if not (((p^5 mod 11) = 1) and ((p mod 3) = 1)) then
    return "Bad choice of p";
fi;
l:= List([0..p-1], i -> Z(p)^i);
cand:= Filtered(l, i -> i^2 + i + 3 = 0*Z(p));
c:= cand[1];;
C:= -1*Z(p)^0 - c;;
m1:= [[0*Z(p),Z(p)^0,0*Z(p),0*Z(p),0*Z(p)],
  [Z(p)^0,0*Z(p),0*Z(p),0*Z(p),0*Z(p)],
  [0*Z(p),0*Z(p),0*Z(p),0*Z(p),Z(p)^0],
  [Z(p)^0,-Z(p)^0,c,Z(p)^0,-c],
  [0*Z(p),0*Z(p),Z(p)^0,0*Z(p),0*Z(p)]];;
m2:= [[0*Z(p),0*Z(p),0*Z(p),Z(p)^0,0*Z(p)],
  [0*Z(p),0*Z(p),Z(p)^0,0*Z(p),0*Z(p)],
  [0*Z(p),-Z(p)^0,-Z(p)^0,0*Z(p),0*Z(p)],
  [-C,0*Z(p),0*Z(p),-Z(p)^0,-c-2*C],
  [Z(p)^0,0*Z(p),0*Z(p),-Z(p)^0,Z(p)^0]];;
H:= Group(Projectivity(m1, GF(p)), Projectivity(m2, GF(p)));
mat:= m1*m2*m1*m2*m1*m2*m1*m2*m2;;
eigs:= Eigenvectors(GF(p), mat);
PS:= PG(4,p);;
vects:= List(eigs, i -> VectorSpaceToElement(PS,i));
orbs:= DuplicateFreeList(List(vects, i -> AsSet(FiningOrbit(H,i))));
list:= [];;
for o in orbs do
    if Length(o) = 110 then
        Add(list, [SpanSizes(o,5), FiningOrbit(H,o[1])]);
    fi;
od;
return list;
end;
\end{verbatim}
\end{func}

\begin{func}
\label{func:110arcsp}
The following function checks which entries $p$ in a list of primes are such that $\PG(4,p)$ contains a $\PSL(2,11)$-transitive $110$-arc.
\begin{verbatim}
110arcsp:= function(list)

local p, l, i, temp;

l:= [];;
for p in list do    
        temp:= 110arcp(p);
        Print([p, temp], "\n");
        for i in temp do
            if Length(i[1][1]) = 1 then 
                Add(l, p);
            fi;
        od;    
od;
return l;
end;  
 \end{verbatim}

\end{func}

\begin{func}
\label{func:110arc2}
The following function checks whether $\PG(4,p^2)$ contains a $\PSL(2,11)$-transitive $110$-arc, where $p$ is a prime.

\begin{verbatim}
110arc:= function(p)

local l, cand, c, C, m1, m2, mat, eigs, test, pos, PS, P, v, H, o, list, i, 
vects, orbs;

if not ((p^5 mod 11) = 1) then
    return "Bad choice of p";
fi;
l:= List([0..p-1], i -> Z(p)^i);
cand:= Filtered(l, i -> i^2 + i + 3 = 0*Z(p));
c:= cand[1];;
C:= -1*Z(p)^0 - c;;
m1:= [[0*Z(p),Z(p)^0,0*Z(p),0*Z(p),0*Z(p)],
  [Z(p)^0,0*Z(p),0*Z(p),0*Z(p),0*Z(p)],
  [0*Z(p),0*Z(p),0*Z(p),0*Z(p),Z(p)^0],
  [Z(p)^0,-Z(p)^0,c,Z(p)^0,-c],
  [0*Z(p),0*Z(p),Z(p)^0,0*Z(p),0*Z(p)]];;
m2:= [[0*Z(p),0*Z(p),0*Z(p),Z(p)^0,0*Z(p)],
  [0*Z(p),0*Z(p),Z(p)^0,0*Z(p),0*Z(p)],
  [0*Z(p),-Z(p)^0,-Z(p)^0,0*Z(p),0*Z(p)],
  [-C,0*Z(p),0*Z(p),-Z(p)^0,-c-2*C],
  [Z(p)^0,0*Z(p),0*Z(p),-Z(p)^0,Z(p)^0]];;
H:= Group(Projectivity(m1, GF(p^2)), Projectivity(m2, GF(p^2)));
mat:= m1*m2*m1*m2*m1*m2*m1*m2*m2;;
eigs:= Eigenvectors(GF(p^2), mat);
PS:= PG(4,p^2);;
vects:= List(eigs, i -> VectorSpaceToElement(PS,i));
orbs:= DuplicateFreeList(List(vects, i -> AsSet(FiningOrbit(H,i))));
list:= [];;
for o in orbs do
    if Length(o) = 110 then
        Add(list, [SpanSizes(o,5), FiningOrbit(H,o[1])]);
    fi;
od;
return list;
end;
\end{verbatim}
\end{func}

\begin{func}
\label{func:110arcs2}
The following function checks which entries $p$ in a list of primes are such that $\PG(4,p^2)$ contains a $\PSL(2,11)$-transitive $110$-arc.
\begin{verbatim}
110arcs:= function(list)

local p, l;

l:= [];;
for p in list do    
        temp:= 110arc(p);
        Print([p, temp], "\n");
        for i in temp do
            if Length(i[1][1]) = 1 then 
                Add(l, p);
            fi;
        od;    
od;
return l;
end;
\end{verbatim}
\end{func}

\begin{func}
\label{func:60arc}
The following function checks whether $\PG(4,p)$ contains a $\PSL(2,11)$-transitive $60$-arc, where $p$ is a prime.
\begin{verbatim}
60arc:= function(p)

local l, cand, c, C, m1, m2, mat, eigs, test, pos, PS, P, v, H, 
o, list, i, vects, orbs;

if not ((p mod 11) = 1) then
    return "Bad choice of p";
fi;
l:= List([0..p-1], i -> Z(p)^i);
cand:= Filtered(l, i -> i^2 + i + 3 = 0*Z(p));
c:= cand[1];;
C:= -1*Z(p)^0 - c;;
m1:= [[0*Z(p),Z(p)^0,0*Z(p),0*Z(p),0*Z(p)],
  [Z(p)^0,0*Z(p),0*Z(p),0*Z(p),0*Z(p)],
  [0*Z(p),0*Z(p),0*Z(p),0*Z(p),Z(p)^0],
  [Z(p)^0,-Z(p)^0,c,Z(p)^0,-c],
  [0*Z(p),0*Z(p),Z(p)^0,0*Z(p),0*Z(p)]];;
m2:= [[0*Z(p),0*Z(p),0*Z(p),Z(p)^0,0*Z(p)],
  [0*Z(p),0*Z(p),Z(p)^0,0*Z(p),0*Z(p)],
  [0*Z(p),-Z(p)^0,-Z(p)^0,0*Z(p),0*Z(p)],
  [-C,0*Z(p),0*Z(p),-Z(p)^0,-c-2*C],
  [Z(p)^0,0*Z(p),0*Z(p),-Z(p)^0,Z(p)^0]];;
H:= Group(Projectivity(m1, GF(p)), Projectivity(m2, GF(p)));
mat:= m1*m2;;
eigs:= Eigenvectors(GF(p), mat);
PS:= PG(4,p);;
vects:= List(eigs, i -> VectorSpaceToElement(PS,i));
orbs:= DuplicateFreeList(List(vects, i -> AsSet(FiningOrbit(H,i))));
list:= [];;
for o in orbs do
    if Length(o) = 60 then
        Add(list, [SpanSizes(o,5), FiningOrbit(H,o[1])]);
    fi;
od;
return list;
end;
\end{verbatim}
\end{func}

\begin{func}
\label{func:60arcs}
The following function checks which entries $p$ in a list of primes are such that $\PG(4,p)$ contains a $\PSL(2,11)$-transitive $60$-arc.
\begin{verbatim}
60arcs:= function(list)

local p;

l:= [];;
for p in list do    
        temp:= 60arc(p);
        Print([p, temp], "\n");
        for i in temp do
            if Length(i[1][1]) = 1 then 
                Add(l, p);
            fi;
        od;    
od;
return l;
end;    
\end{verbatim}
\end{func}

\begin{calc}
\label{calc:170}
The following calculation demonstrates that the union of a $60$-arc and a $110$-arc will be a $170$-arc for infinitely many primes $p$. 
\begin{verbatim}
gap> c := E(11)+E(11)^3+E(11)^4+E(11)^5+E(11)^9; C := -1-c;
E(11)+E(11)^3+E(11)^4+E(11)^5+E(11)^9
E(11)^2+E(11)^6+E(11)^7+E(11)^8+E(11)^10
gap>  A:= [[0,1,0,0,0],
> [1,0,0,0,0],
> [0,0,0,0,1],
> [1,-1,c,1,-c],
> [0,0,1,0,0]];
[ [ 0, 1, 0, 0, 0 ], [ 1, 0, 0, 0, 0 ], [ 0, 0, 0, 0, 1 ], 
[ 1, -1, E(11)+E(11)^3+E(11)^4+E(11)^5+E(11)^9, 1, 
-E(11)-E(11)^3-E(11)^4-E(11)^5-E(11)^9 ],[ 0, 0, 1, 0, 0 ] ]
gap>  A:= [[0,1,0,0,0],
> [1,0,0,0,0],
> [0,0,0,0,1],
> [1,-1,c,1,-c],
> [0,0,1,0,0]];
[ [ 0, 1, 0, 0, 0 ], [ 1, 0, 0, 0, 0 ], [ 0, 0, 0, 0, 1 ], 
[ 1, -1, E(11)+E(11)^3+E(11)^4+E(11)^5+E(11)^9, 1, 
-E(11)-E(11)^3-E(11)^4-E(11)^5-E(11)^9 ], [ 0, 0, 1, 0, 0 ] ]
gap>  B:= [[0,0,0,1,0],
> [0,0,1,0,0],
> [0,-1,-1,0,0],
> [-C,0,0,-1,-c-2*C],
> [1,0,0,-1,1]];
[ [ 0, 0, 0, 1, 0 ], [ 0, 0, 1, 0, 0 ], [ 0, -1, -1, 0, 0 ], 
[ -E(11)^2-E(11)^6-E(11)^7-E(11)^8-E(11)^10, 0, 0, -1, 
-E(11)-2*E(11)^2-E(11)^3-E(11)^4-E(11)^5-2*E(11)^6-2*E(11)^7
-2*E(11)^8-E(11)^9-2*E(11)^10 ], 
  [ 1, 0, 0, -1, 1 ] ]
gap>  G:= Group(A,B);
<matrix group with 2 generators>
gap> Order(G);
660
gap> M6:= A*B*A*B*A*B*A*B*B;;
gap> Order(M6);
6
gap> M11:=A*B;;  
gap> Order(M11);
11
gap> F:= Field(E(33));;
gap> eigvals11:= Eigenvalues(F,M11);
[ E(11), E(11)^3, E(11)^4, E(11)^5, E(11)^9 ]
gap> eigvals6:=Eigenvalues(F,M6);
[ 1, E(3), E(3)^2, -E(3)^2, -E(3) ]
gap> eigs6:= Eigenvectors(F,M6);;
gap> eigs11:= Eigenvectors(F,M11);;
gap> H6:=Group(M6);     
<matrix group with 1 generators>
gap> Order(H6);
6
gap> reps110:= List(RightCosets(G,H6), Representative);;
gap> Length(reps110);
110
gap> H11:=Group(M11);     
<matrix group with 1 generators>
gap> Order(H11);
11
gap> reps60:= List(RightCosets(G,H11), Representative);;
gap> Length(reps60); 
60
gap> set110:= List(reps, i -> eigs6[2]*i);; 
gap> set110:= List(reps110, i -> eigs6[2]*i);;
gap> set60:= List(reps60, i -> eigs11[1]*i);; 
gap> set:=[];; set:=Concatenation(set60,set110);
gap> IsDuplicateFreeList(set);
true
gap> Length(set);                           
170
gap> preset:= List([2..170], i -> set[i]);;
gap> precomb:= Combinations(preset, 4);;
gap> Length(precomb);
32795126
gap> sets:= List(precomb, i -> Concatenation([set[1]], i));;
gap> dets170:= List(sets, i -> DeterminantMat(i));;
gap> Length(dets170);
32795126
gap> 0 in dets170;
false
\end{verbatim}
\end{calc}

\end{document}